\documentclass[a4paper]{article}

\usepackage{amsfonts}
\usepackage{amssymb,amsthm,amsmath,graphicx,bm,ebezier}
\usepackage{hyperref}
\usepackage{caption}
\usepackage[normalem]{ulem}
\usepackage{url,lineno}
\usepackage[utf8]{inputenc}
\usepackage{arydshln}
\usepackage[ruled,vlined,linesnumbered]{algorithm2e}
\usepackage[noend]{algpseudocode}
\usepackage{lscape}
\usepackage{multirow}
\usepackage[dvipsnames]{xcolor}
\usepackage{stmaryrd}
\usepackage{tikz-qtree}

\usepackage{tikz}
\usetikzlibrary{automata, positioning}

\usepackage{todonotes}

\newtheorem{theorem}{Theorem}
\newtheorem{proposition}[theorem]{Proposition}

\newtheorem{lemma}[theorem]{Lemma}

\newtheorem{example}[theorem]{Example}
\newtheorem{remark}[theorem]{Remark}
\newtheorem{definition}[theorem]{Definition}

\def\CaH{\mathcal{H}}
\def\CaC{\mathcal{C}}
\def\CaS{\mathcal{S}}

\def\N{\mathbb{N}}

\def\msg{\operatorname{msg}}
\def\g{\operatorname{g}}
\def\m{\operatorname{m}}
\def\n{\operatorname{n}}
\def\opN{\operatorname{N}}
\def\Fb{\operatorname{Fb}}
\def\SG{\operatorname{SG}}
\def\Ap{\operatorname{Ap}}

\title{On some affine semigroups characterized by a finite-state automata}

\date{}
\author{J. I. Farrán, J. C. Rosales, R. Tapia-Ramos, and A. Vigneron-Tenorio}

\begin{document}
\maketitle
\abstract{
This work introduces a new kind of affine semigroups called $P$-semigroups. Within the framework of $\CaC$-semigroups, we define a finite-state automaton associated to them. Moreover, this automaton determines whether a $\CaC$-semigroup is a $P$-semigroup, which represents a bridge between affine semigroups and Discrete Mathematics. Furthermore, some algorithms for computing all the $P$-semigroups with a fixed Frobenius element, genus, or multiplicity are provided.

{\small

}
}

\section*{Introduction}

Let $\N$, $\mathbb{Q}$, and $\mathbb{R}$ be the sets of non-negative integers, rationals, and real numbers, respectively.
For any $d\in \N$, let $\CaC\subseteq \N^d$ be a non-negative integer cone finitely generated, and assume that it has at least $d$ extremal rays.
In general, a monoid $S\subseteq \N^d$ is a semigroup (that is, a non-empty set closed under the usual addition in $\N^d$) containing the zero element. A finitely generated monoid is called an affine semigroup.
If the monoid $S$ is a subset of $\CaC$, and $\CaC\setminus S$ is finite, then $S$ is finitely generated and is called a $\CaC$-semigroup. This class of semigroups naturally generalizes numerical semigroups and, in particular, includes the class of generalized numerical semigroups first introduced in \cite{GenSemNp}.
Hence, concepts such as the set of gaps of $S$ ($\CaC\setminus S$), and its genus (cardinality of $\CaC\setminus S$) are defined as extensions of the corresponding ones for numerical semigroups (\cite{libroRosales}).
Moreover, once a total order on $\N^d$ is given, the Frobenius element of $S$, and the multiplicity of $S$, are defined as the maximum element in $\CaC\setminus S$ and as the minimum element in $S\setminus\{0\}$, respectively. In this work, a total order $\preceq$ in $\N ^d$ is fixed. A more detailed discussion on $\CaC$-semigroups is provided in \cite{Csemigroup}.

Given a non-empty finite set $P\subset \N^d$, this work introduces a family of affine semigroups denoted by $P$-semigroups. We say that an affine semigroup is a $P$-semigroup if $\left(\{s\}+P\right)\cap S\ne \emptyset$, for all $s\in S\setminus\{0\}$. Trivially, when $0\in P$, any affine semigroup in $\N^d$ is a $P$-semigroup. So, assume that the zero element does not belong to $P$.
In this setting, we can associate to each $\CaC$-semigroup a certain automaton. In particular,
we characterize when a $\CaC$-semigroup is a $P$-semigroup in terms of the language that is recognized by this automaton (Theorem \ref{MainTheoremAutomata}). This result provides an unexpected link between Semigroup Theory and Automata theory.

We recall that (deterministic) finite state automata, as a computation model in Computer Science, are abstract machines having a finite number of states so that, starting at an initial state, the current state changes for a given Input following a transition function. In this way, a string of Inputs is recognized by the automaton if the final state is considered valid by this model, and the set of recognized strings is just the (formal) language recognized by the machine (see \cite{Meduna} for further details).

In this work, in addition to introducing the concept of $P$-semigroups, we also study several properties of $P$-semigroups. The obtained results allow us to design and implement some algorithmic methods to compute all the $P$-semigroups with a given Frobenius element, genus, or multiplicity for a fixed integer cone, a set $P$, and a monomial order.

The content of this work is organized as follows. Section \ref{Sec1} provides the necessary background on affine semigroups and includes a characterization of $P$-semigroups from the minimal generating set of an affine semigroup.
In Section \ref{Sec2}, we recall basic notions of automata theory and show how a specifically defined automaton serves to detect whether a $\CaC$-semigroup satisfies the $P$-semigroup condition.
Sections \ref{Sec3}, \ref{Sec4}, and \ref{Sec5} present several results on $P$-semigroups, together with some algorithms to compute all the $P$-semigroups satisfying some prescribed properties.

\section{Preliminaries and affine $P$-semigroups}\label{Sec1}

A real cone in $\mathbb{R}^d$ is the intersection of finitely many linear closed half-spaces. This set can also be defined from a set of vectors in $\mathbb{R}^d$, that is, a real cone is the set $\{\sum_{i=1}^n a_iv_i \mid a_i\in \mathbb{R}_+\}$ where $\{v_1,\ldots ,v_n\}\subset \mathbb{R}^d$ ($\mathbb{R}_+$ corresponds with the set of non-negative real numbers). We consider that an integer non-negative cone $\CaC\subseteq \N^d$ is the affine monoid given by $\{\sum_{i=1}^n a_iv_i \mid a_i\in \mathbb{R}_+\}\cap \N^d$ where $\{v_1,\ldots ,v_n\}\subset \mathbb{Q_+}^d$.
Since $\{v_1,\ldots ,v_n\}\subset \mathbb{Q_+}^d$, the integer cone $\CaC$ is finitely generated (see \cite{Bruns}). Hence, any $\CaC$-semigroup is also a finitely generated semigroup. Recall that, in this work, we also assume that the cone $\CaC$ has at least $d$ extreme rays.

It is well known that any $\CaC$-semigroup $S$ admits a unique minimal generating set, denoted by $\msg(S)$. Any element belonging to $\CaC\setminus S$ is called a gap of $S$, and the set of all gaps of $S$ is denoted by $\CaH(S)$. An important invariant related to $S$ is its genus, which is defined as $\g(S)=\sharp \big(\CaH(S) \big)$, where $\sharp(L)$ is the cardinality of any set $L$.
A gap $x$ of $S$ is called special gap of $S$ if $2x \in S$ and $x + s \in S$, for all $s \in S \setminus \{0\}$. We denote by $\SG(S)$ the set of all special gaps of $S$.

In this work, we fix a monomial order $\preceq$ on $\mathbb N^d$, that is, a total order compatible with addition, where $0 \preceq x$ for any $x\in \mathbb N^d$ (see \cite{Cox}). With respect to this fixed order $\preceq$ on $\N^d$, the Frobenius element of $S$ is $\Fb(S)=\max_\preceq \CaH(S)$. When $\CaH(S)$ is empty, $\Fb(S)=(-1,-1,\dots ,-1)\in \mathbb{Z}^d$. The multiplicity of $S$ is the element $\m(S)=\min_\preceq \big(S\setminus\{0\}\big)$, or equivalently, $\m(S)=\min_\preceq \big(\msg(S)\big)$. We say that an element $s$ of $S$ is a small element when $s\prec \Fb(S)$. The set of all small elements of $S$ is denoted by $\opN(S)$.
Obviously, $\Fb(S)$, $\g(S)$ and $\opN(S)$ depend on the fixed order.

Let $P$ be a non-empty subset of $\N^d$. Recall that a $P$-semigroup is an affine semigroup satisfying $\left(\{s\}+P\right)\cap S\ne \emptyset$, for all $s\in S\setminus\{0\}$. The set of all $P$-semigroups is denoted by $\CaS(P)$.

The following result provides a characterization of $P$-semigroups, which can be used to check computationally whether an affine semigroup is a $P$-semigroup.

\begin{proposition}
Let $S$ be an affine semigroup. Then, $S$ is a $P$-semigroup if and only if  $\left(\{a\}+P\right)\cap S\ne \emptyset$, for all $a\in \msg(S)$.
\end{proposition}

\begin{proof}
The direct implication is immediate. Conversely, let $s$ be a non-zero element of $S$. Then, there exist $a \in \msg(S)$ and $s' \in S$ such that $s = a + s'$. By hypothesis, there exists $p \in P$ such that $a+p \in S$. It follows that $s+p = (a+s') + p \in S$. Hence, $S$ is a $P$-semigroup.
\end{proof}

\begin{example}\label{ej:P-smgp}
Consider the non-negative integer cone $\CaC$ spanned by
$$\{(1,1), (1,2), (1,3), (2,1)\},$$
and let $P=\{(1,2), (2,0)\}$. Define the $\CaC$-semigroup $S$ minimally generated by
\begin{multline*}
 \msg(S)=   \{(1, 1), (2, 3), (3, 8), (3, 9), (4, 3), (4, 7), (4, 8), (4, 11), (4, 12),\\ (5, 14), (5, 15), (6, 4), (7, 4), (8, 4), (10, 5), (12, 6), (14, 7)\},
\end{multline*}
whose set of gaps is
\begin{multline*}
    \CaH(S)=\{(1, 2), (2, 1), (1, 3), (3, 2), (2, 4), (4, 2), (2, 5), (2, 6), (3, 5), (5, 3),\\ (3, 6), (5, 4), (6, 3), (3, 7)\}.
\end{multline*}
Using Figure \ref{fig:Psemig}, it can be checked graphically that $S$ is a $P$-semigroup. The empty circles correspond to its set of gaps, the blue squares are its minimal generators, and the red circles are elements belonging to it.
    \begin{figure}[ht]
    \centering
    \includegraphics[scale=.27]{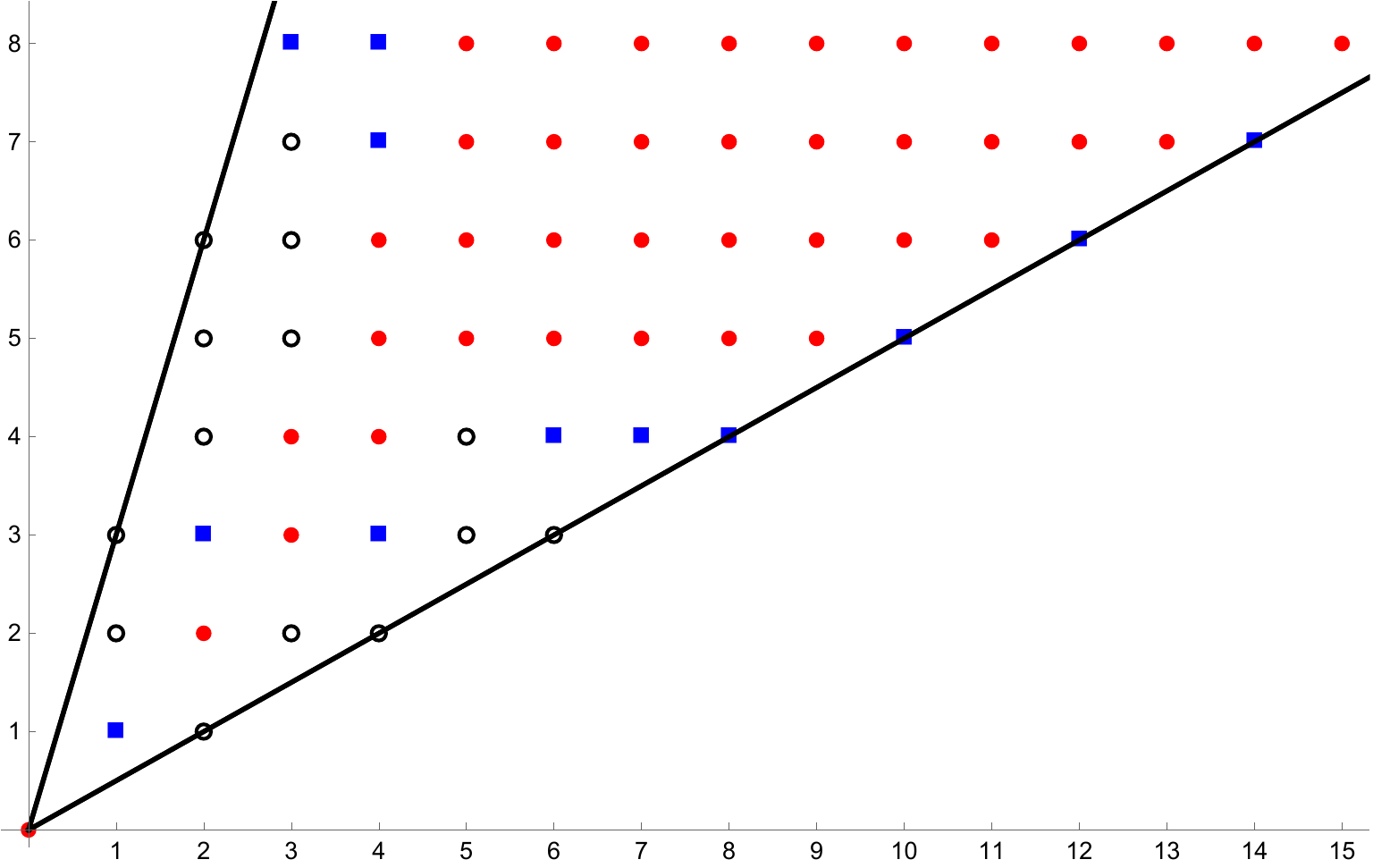}
    \caption{A $P$-semigroup}
    \label{fig:Psemig}
\end{figure}

\end{example}

\section{Automata and $P$-semigroups}\label{Sec2}

Recall that the notion of being a $P$-semigroup is defined for affine semigroups. From now on, we focus on $\CaC$-semigroups.
In this section, we introduce the basic notions of automata theory and define the automaton associated to a $\CaC$-semigroup. More details on automata and formal languages can be seen in \cite{Meduna}.

\begin{definition}
A (deterministic) finite-state machine (automaton in short) is a tuple $M=(\Sigma, {\cal A}, \sigma_{0}, F, f)$ where:
\begin{itemize}
\item $\Sigma$ is a finite set of states.
\item ${\cal A}$ is a finite alphabet of symbols.
\item $\sigma_{0}$ is the initial state.
\item $F\subseteq\Sigma$ is the set of accepted final states.
\item A transition function $f:\Sigma\times{\cal A}\rightarrow\Sigma$.
\end{itemize}
\end{definition}

By concatenation of symbols we get the set ${\cal A}^{\ast}$ of strings, and by iterating the transition $f$, one can extend $f$ on strings
\[
f^{\ast}:\Sigma\times{\cal A}^{\ast}\rightarrow\Sigma.
\]
In this way, a string $s\in{\cal A}^{\ast}$ is said to be recognized by the automaton $M$ if $f^{\ast}(\sigma_{0},s)\in F$. Otherwise, $s$ is rejected by $M$.

The set $L(M)\subseteq{\cal A}^{\ast}$ of all the strings (or words) recognized by the automaton $M$ is called the (formal) language recognized by $M$.

By convention, if $\varepsilon$ is the empty string, we set $f^{\ast}(\sigma_{0},\varepsilon)=\sigma_{0}$, so that $\varepsilon$ is accepted by the automaton if and only if $\sigma_{0}\in F$.

On the other hand, $\sigma\in \Sigma$ is said to be a dead state if $f(\sigma,\alpha)=\sigma$ for every $\alpha\in{\cal A}$.

Finally, an automaton is said to be connected if every state in $\Sigma$ can be reached from the initial state $\sigma_{0}$ for a particular string $s\in{\cal A}^{\ast}$.

Now we define the automaton associated to a $\CaC$-semigroup. Let $P\subset\N^{d}$ and $S$ be a $\CaC$-semigroup.
To define the automaton associated to $S$, it is necessary to assume finiteness of the set $\opN(S)$. From this fact, the fixed total order $\preceq$ have to satisfy that $\{m\in \CaC \mid m\prec f \}$ is finite, for any $f\in \CaC$. For instance, a graded monomial order is a total order of this kind (see \cite{Cox}).
The underlying idea is that, starting from any element of $S$ and moving only along the directions in $P$, we remain inside $S$. This observation inspires the definition of the following automaton. Denote by $G$ the set of minimal generators of $S$ that belong to $\opN(S)$, and let $Q_p$ be the set $\{s\in \opN(S)\setminus\{0\} \mid s-p\notin S, \text{ for all } p\in P\}$. This last concept recalls the notion given in \cite{R-G-U-Cohen-Macauly}. The automaton associated to $S$ works as follows.

\begin{definition}
Let $S$ be a $\CaC$-semigroup. The automaton associated to $S$ is the tuple $M(S)=(\Sigma, {\cal A}, \sigma_{0}, F, f)$ where:
\begin{itemize}
\item $\Sigma=\opN(S)\cup\{\kappa,\chi\}$, where $\kappa$ represents any element of $\CaC$ strictly greater than $\Fb(S)$, and $\chi$ denotes a dead state.
\item ${\cal A}=P\cup G\cup Q_{P}$ (note that this union may not be disjoint).
\item $\sigma_{0}=0\in S$.
\item $F=\Sigma\setminus\{\sigma_{0},\chi\}$.
\item The transition function $f$ defined as follows:
\[
f(0,\alpha)=\left\{
\begin{array}{ll}
\alpha & \mbox{if $\alpha\in \opN(S)\setminus\{0\}$}\\
\kappa & \mbox{else if $\alpha\in S$}\\
\chi & \mbox{otherwise}
\end{array}
\right.
\]
For $q\in \opN(S)\setminus\{0\}$, then
\[
f(q,\alpha)=\left\{
\begin{array}{ll}
q+\alpha & \mbox{if $\alpha\in P$ and $q+\alpha\in \opN(S)$}\\
\kappa & \mbox{else if $\alpha\in P$ and $q+\alpha\in S$}\\
\chi & \mbox{otherwise}
\end{array}
\right.
\]
Finally,
\[
f(\kappa,\alpha)=\left\{
\begin{array}{ll}
\kappa & \mbox{if $\alpha\in P$}\\
\chi & \mbox{otherwise}
\end{array}
\right.
\]
and $f(\chi,\alpha)=\chi$ for all $\alpha\in{\cal A}$.
\end{itemize}
\end{definition}

\begin{remark}
Notice that the set $Q_{P}$ is necessary, since otherwise the automaton may not be connected.
\end{remark}

\begin{example}
Consider the $P$-semigroup $S$ given in Example \ref{ej:P-smgp}, and let $\preceq$ be the graded lexicographical order. So, $\Fb(S)=(3,7)$. In this situation, the elements of the automaton $M=(\Sigma, {\cal A}, \sigma_{0}, F, f)$ associated to $S$ are the following:
\begin{itemize}
\item $\Sigma=\{(0,0), (1,1), (2,2), (2,3), (3,3), (3,4), (4,3), (4,4), (4,5)\}\cup\{\kappa,\chi\}$.
\item ${\cal A}=\{(1,1), (1,2), (2,0), (2,2), (2,3), (3,3), (4,3), (4,4) \}$.
\item $\sigma_{0}=(0,0)$.
\item $F=\Sigma \setminus\{\sigma_{0}, \chi \}$.
\item A transition function $f$ is given in Table \ref{tab:TablaEstados}.
\end{itemize}

\begin{table}[ht]
    \centering
    \begin{tabular}{||c|c|c|c|c|c|c|c|c||}
    \hline
    State
     & \multicolumn{8}{|c||}{$f$}\\ 
          & \multicolumn{8}{|c||}{Input}\\
          & $(1,1)$ & $p_1=(1,2)$ & $p_2=(2,0)$ & $(2,2)$ & $(2,3)$ & $(3,3)$ & $(4,3)$ & $(4,4)$ \\
          \cline{1-9}
$\sigma_{0}=(0,0)$         & $s_1$ & $\chi$ & $\chi$ & $s_2$ & $s_3$ & $s_4$ & $s_6$ & $s_7$ \\
$s_1=(1,1)$         & $\chi$ & $s_3$ & $\chi$  & $\chi$  & $\chi$  & $\chi$  & $\chi$  & $\chi$  \\
$s_2=(2,2)$         & $\chi$ & $s_5$ & $\chi$ & $\chi$  & $\chi$  & $\chi$  & $\chi$  & $\chi$  \\
$s_3=(2,3)$         & $\chi$ & $\chi$ & $s_6$ & $\chi$  & $\chi$  & $\chi$  & $\chi$  & $\chi$  \\
$s_4=(3,3)$         & $\chi$ & $s_8$ & $\chi$ & $\chi$  & $\chi$  & $\chi$  & $\chi$  & $\chi$  \\
$s_5=(3,4)$         & $\chi$ & $\kappa$ & $\chi$ & $\chi$  & $\chi$  & $\chi$  & $\chi$  & $\chi$  \\
$s_6=(4,3)$         & $\chi$ & $\kappa$ & $\chi$  & $\chi$  & $\chi$  & $\chi$  & $\chi$  & $\chi$  \\
$s_7=(4,4)$         & $\chi$ & $\kappa$ & $\kappa$  & $\chi$  & $\chi$  & $\chi$  & $\chi$  & $\chi$  \\
$s_8=(4,5)$         & $\chi$ & $\kappa$ & $\kappa$  & $\chi$  & $\chi$  & $\chi$  & $\chi$  & $\chi$  \\

$\kappa$        & $\chi$ & $\kappa$ & $\kappa$  & $\chi$  & $\chi$  & $\chi$  & $\chi$  & $\chi$  \\
$\chi$         & $\chi$ & $\chi$  & $\chi$  & $\chi$  & $\chi$  & $\chi$  & $\chi$  & $\chi$  \\

        \hline
    \end{tabular}
    \caption{The state table for the automaton}
    \label{tab:TablaEstados}
\end{table}

Figure \ref{fig:Path} illustrates a path inside $S$ following the directions in $P$, corresponding to
the string $s_{1}p_{1}p_{2}p_{1}$, that is recognized by the automaton.
    \begin{figure}[ht]
    \centering
    \includegraphics[scale=.4]{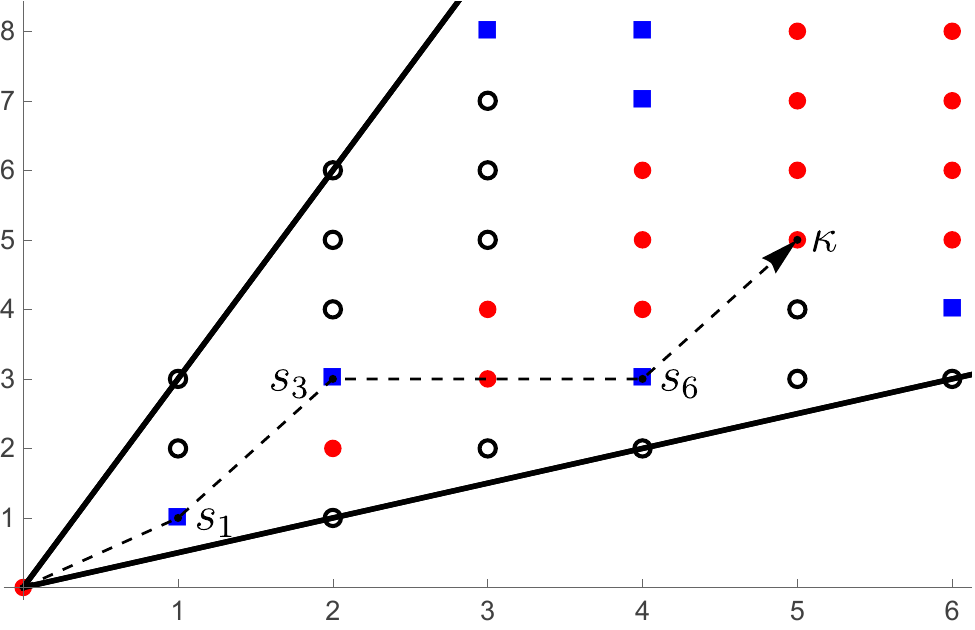}
    \caption{A path inside the $P$-semigroup}
    \label{fig:Path}
\end{figure}

\end{example}

The following result is a direct consequence of the definition of the automaton $M(S)$.

\begin{lemma}
The language recognized by $M(S)$ is contained in the set of all the strings $\alpha_1\alpha_2\cdots\alpha_n$ such that $\alpha_1\in{\cal A}\setminus P$, and $\alpha_i\in P$ for all $i>1$. In fact, the words recognized by $M(S)$ correspond to paths inside the cone $\CaC$ starting at points in $G\cup Q_{P}$, and following directions in $P$, which always keep inside $S$.
\end{lemma}

The following result states that $M(S)$ detects if $S$ is actually a $P$-semigroup.

\begin{theorem}\label{MainTheoremAutomata}
A $\CaC$-semigroup $S$ is a $P$-semigroup if and only if every string $\alpha_1\alpha_2\cdots\alpha_n$ recognized by $M(S)$ can be extended to a
string $\alpha_1\alpha_2\cdots\alpha_n\alpha_{n+1}$ also recognized by $M(S)$.
\end{theorem}

\begin{proof}

It is clear that if $S$ is a $P$-semigroup and the string is recognized by $M(S)$, one can choose a suitable symbol $\alpha\in P$ that, when applied to the accepted state $\sigma_n$, leads again to an accepted state.

Conversely, if $S$ is not a $P$-semigroup, then there exists an element $s\in \opN(S)\setminus\{0\}$ such that $s+p\not\in S$ for all $p\in P$.
Since the automaton is connected, there exists a string $\alpha=\alpha_1\alpha_2\cdots\alpha_n$ recognized by $M(S)$ such that $f(0,\alpha)=s$, and this string cannot be extended to any recognized string.

\end{proof}

\section{The elements of $\CaS(P)$ with a given genus}\label{Sec3}

Let $P\subset \N^d\setminus\{0\}$ be a finite set. Recall that the set of all $P$-semigroups is denoted by $\CaS(P)$.
In this section, we compute all $P$-semigroups with a given genus using a rooted directed graph, and we illustrate the procedure with an example. Note that, if $P\cap \CaC=\emptyset$, then it is not guaranteed that the set $\CaS(P)$ is not empty.
So, from now on, we assume that $P \cap \CaC \neq \emptyset$. In this situation, we can state that $\CaS(P)$ is infinite, since for any $f \in \CaC \setminus \{0\}$, the $\CaC$-semigroup $\Delta(f)=\{x\in \CaC \mid x\succ f\}\cup \{0\}$ is a $P$-semigroup. For numerical semigroups, such a semigroup is known as a half-line or an ordinary semigroup. For non-numerical $\CaC$-semigroups, we refer to them as ordinary $\CaC$-semigroups. The terminology is inspired by \cite{R-T-V-Ordinario}, although the authors use the notion of ordinary semigroup based on the conductor rather than the Frobenius element.
The concept of ordinary semigroup adopted here is also not equivalent to the definition presented in \cite{CistoOrdinary}.

Observe that the maximum element of $\CaS(P)$ with respect to the inclusion is $\CaC$. The following lemma is needed for the upcoming definition.

\begin{lemma}\label{lemma:SUfb}
Let $S$ be a $P$-semigroup. If $S\ne \CaC$, then $S\cup \{\Fb(S)\}$ is also a $P$-semigroup.
\end{lemma}

\begin{proof}
Clearly, $S \cup \{\Fb(S)\}$ is a $\CaC$-semigroup. It remains to verify that for every $s \in S \cup \{\Fb(S)\}$ there exists $p \in P$ such that $s+p$ belongs to $S \cup \{\Fb(S)\}$.  If $s \in S$, then $s+p \in S$ for some $p \in P$, since $S \in \CaS(P)$. If $s = \Fb(S)$, then $s+p \succ \Fb(S)$, and thus $s+p \in S \cup \{\Fb(S)\}$.
\end{proof}

Consider the graph $G\left(\CaS(P)\right)$, whose vertex set is $\CaS(P)$, and where a pair $(S,T)\in \CaS^2(P)$ is an edge if $T=S \cup \{\Fb(S)\}$. In this case, we say that $S$ is a child of $T$.  The following result describes the structure of $G\left(\CaS(P)\right)$ and provides an explicit characterization of the children of each vertex, which is a key step in the recursive construction of the graph.

\begin{theorem}\label{thr:tree}
The graph $G\left(\CaS(P)\right)$ is a tree with root $\CaC$. Moreover,  the set of children of any $T\in \CaS(P)$ is given by
\[
\{ T \setminus \{a\}\in \CaS(P) \mid a \in \msg(T),\ a \succ \Fb(T)\}.
\]
\end{theorem}

\begin{proof}
Let $S \in \CaS(P)$ and consider the sequence $\{S_i\}_{i \in \mathbb{N}}$ defined by starting at $S_0 = S$, and for each $i \geq 0$, define $S_{i+1}=S_i\cup\{\Fb(S_i)\}$  if $S_i\ne\CaC$, and $S_{i+1}=\CaC$, otherwise.  Clearly, this sequence stabilizes at $S_{\g(S)} = \mathcal{C}$. By the uniqueness of the Frobenius element, we conclude that $G\left(\CaS(P)\right)$ is a tree with root $\CaC$.

Now, assume that $S$ is a child of $T$ in $G\left(\CaS(P)\right)$, that is, $T = S \cup \{\Fb(S)\} \in \CaS(P)$. So, $S = T \setminus \{a\} \in \CaS(P)$, where $a = \Fb(S)$, which implies that $a \in \msg(T)$ and $a \succ \Fb(T)$.
Conversely, let $T \setminus \{a\} \in \CaS(P)$ for some $a \in \msg(T)$ satisfying $a \succ \Fb(T)$. Then, $(T \setminus \{a\}, T) \in \CaS(P)^2$ is an edge in $G\left(\CaS(P)\right)$, since  $T = \big(T \setminus \{a\}\big) \cup \big\{\Fb(T \setminus \{a\})\big\}$, and $\Fb(T \setminus \{a\}) = a$.
\end{proof}

Notice that the condition $T \setminus \{a\}\in \CaS(P)$ is not straightforward to check computationally. To address this, the following statement presents a characterization that facilitates the computational verification of membership.

\begin{proposition}\label{prop:caracREMOVEelement}
Let $S\in \CaS(P)$ and $a\in \msg(S)$. Then, $S\setminus\{a\}\notin \CaS(P)$ if and only if there exists $p\in P$ such that $a-p\in S\setminus\{0\}$, and $a-p+p'\notin S$ for all $p'\in P\setminus\{p\}$.
\end{proposition}

\begin{proof}
Assume $S \setminus \{a\} \notin \CaS(P)$. So, there exists $s \in S \setminus \{0,a\}$ such that $(\{s\}+P) \cap (S \setminus \{a\}) = \emptyset$.
Since $S$ is a $P$-semigroup, we deduce that $(\{s\}+P) \cap S = \{a\}$.
Let $p \in P$ with $s+p = a$. Then, $a-p = s \in S \setminus \{0\}$.
In particular, $a-p+p' \notin S$ for all $p' \in P \setminus \{p\}$.
Conversely, clearly $a-p\in S\setminus\{a,0\}$ and $(\{a-p\}+P)\cap (S\setminus\{a\})=\emptyset$.
\end{proof}

Theorem \ref{thr:tree} can be applied repeatedly without limitation, and the set $\CaS(P)$ contains infinitely many elements. However, from a computational perspective, it is not feasible without imposing restrictions.  Therefore, to design an algorithm, we fix the genus as indicated.

\begin{algorithm}[h]
\caption{Computing $P$-semigroups with genus $g$.}\label{ComputeCaA(P)fixgenus}
\KwIn{A non-negative integer cone $\CaC$, a finite subset $P$, a monomial order $\preceq$, and a positive integer $g$.}
\KwOut{The set  $\{S\in \CaS(P)\mid g(S)=g\}$.}
\If {$g=0$}
    {\Return{$\CaC$}}
$X \leftarrow \{\CaC\}$\; \label{lineROOT}
\For{$1\leq i \leq g$}{
    $Y \leftarrow \emptyset$\;
    \While {$X\ne\emptyset$}{
        $T \leftarrow \text{First}(X)$\;
        $A \leftarrow \{x\in \msg(T)\mid x\succ \Fb(T)\}$\; \label{lineMSG}
        $B \leftarrow A$\;
        \While {$B \ne\emptyset$}{
           $a \leftarrow \text{First}(B)$\;
           \If{$a-p\in S$ for some $p\in P$, and $a-p+p'\notin S$ for all $p'\in P\setminus\{p\}$}{
               $A \leftarrow A\setminus \{a\}$\;
           }
           $B \leftarrow B\setminus \{a\}$\;
        }
        $Y \leftarrow Y\cup\{T\setminus\{x\}\mid x\in A\}$\;
        $X \leftarrow X\setminus \{T\}$\;
    }
    $X \leftarrow Y$\;
}
    \Return{$X$}
\end{algorithm}

\begin{example}\label{example:treeGenus}
Let $\CaC$ be the non-negative integer cone delimited by the extreme rays generated by $\{(5,1),(3,1)\}$. Consider $P=\{(1,4),(3,1)\}$, and the graded lexicographical order. The tree given in Figure \ref{fig:treeGenus} shows all the $P$-semigroups up to genus 3, where each node represents a $P$-semigroup. The label of a non-root node corresponds to the minimal generator removed in each loop of the algorithm. Since the root has genus zero, the first level of the tree is the set of $P$-semigroups with genus one, and so on.
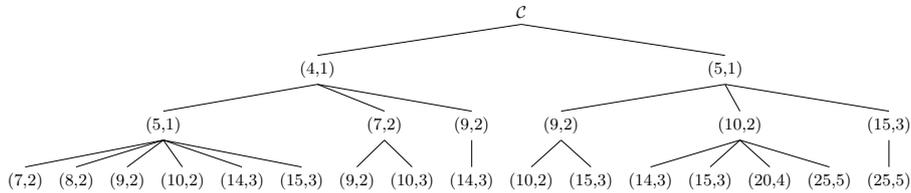
\begin{figure}[ht]
\resizebox{\textwidth}{!}{%
\begin{tikzpicture}
\tikzset{level distance=3em}
\small
\Tree
[.$\CaC$
        [.(4,1)
            [.(5,1) (7,2) (8,2) (9,2) (10,2) (14,3) (15,3) ]
            [.(7,2) (9,2) (10,3) ]
            [.(9,2) (14,3) ]
        ]
        [.(5,1)
            [.(9,2) (10,2) (15,3) ]
            [.(10,2) (14,3) (15,3) (20,4) (25,5) ]
            [.(15,3) (25,5) ]
        ]
]

\end{tikzpicture}
}

\caption{The first three levels of $G\left(\CaS(P)\right)$}
\label{fig:treeGenus}
\end{figure}
\end{example}

\section{The elements of $\CaS(P)$ with a given Frobenius element}\label{Sec4}

Fix $f \in \CaC \setminus \{0\}$. We denote by $\CaS(P,f)$ the set of all $P$-semigroups with Frobenius element equal to $f$.  A simple example of an element in $\CaS(P,f)$ is the $\CaC$-semigroup $\Delta(f)=\{x\in \CaC \mid x\succ f\}\cup\{0\}$, which is the minimal element of $\CaS(P,f)$ with respect to the inclusion.

The main goal of this section is to compute the set $\CaS(P,f)$. To this end, we construct a rooted tree whose vertex set is $\CaS(P,f)$. An observation in this regard is summarized in the next lemma.

\begin{lemma}\label{lemma:Sremovem(S)}
Let $S\in \CaS(P,f)\setminus \{\Delta(f)\}$. Then, $S\setminus\{\m(S)\}\in \CaS(P,f)$.
\end{lemma}

\begin{proof}
Clearly, $S\setminus\{\m(S)\}$ is a $\CaC$-semigroup. Since $S$ is not an ordinary, then $\m(S)\prec \Fb(S)=f$, and thus $\Fb(S\setminus\{\m(S)\})=f$. For any $s \in S \setminus \{0, \m(S)\}$, the $P$-semigroup condition for $S$ ensures the existence of $p \in P$ with $s + p \in S$. Since $s + p \neq \m(S)$, then $s + p \in S \setminus \{\m(S)\}$.
\end{proof}

We define the graph $G\left(\CaS(P,f)\right)$ as follows:
the vertex set is $\CaS(P,f)$, and $(S,T)\in \CaS^2(P,f)$ is an edge if $T = S \setminus \{\m(S)\}$. We now proceed with the central result of this section.

\begin{theorem}\label{thr:treeA(P,f)}
The graph $G\left(\CaS(P,f)\right)$ is a tree with root $\Delta(f)$. Furthermore, the set of children of any $T\in \CaS(P,f)$ is
\[
\{ T \cup \{x\} \mid x \in \SG(T)\setminus\{f\},\; x\prec \m(T), \text{ and } x+p\in T \text{ for some } p\in P\}.
\]
\end{theorem}

\begin{proof}
For any $S \in \CaS(P,f)$, define the sequence $\{S_i\}_{i \in \N}$ where $S_0 = S$, and for each $i \geq 0$, $S_{i+1}=S_i\setminus\{\m(S_i)\}$  if $S_i\ne \Delta(f)$, and $S_{i+1}=\Delta(f)$, otherwise. By Lemma~\ref{lemma:Sremovem(S)}, each $S_i$ belongs to $\CaS(P,f)$.
Let $\n(S)$ be the number of elements of $S$ strictly less than $f$ with respect to $\preceq$. Then, the sequence stabilizes at $S_{\n(S)-1} = \Delta(f)$, and, by using the uniqueness of $\m(S_i)$, we deduce that every $S \in \CaS(P,f)$ admits a unique directed path to the root $\Delta(f)$, and the graph $G(\CaS(P,f))$ is a tree.
For the second statement, suppose that $S=T \cup \{x\}$, with $x = \m(S)$ is a child of $T$. Since $S\in \CaS(P,f)$, then  $x \in \SG(T)\setminus\{f\}$ and at least there exists $p\in P$ such that $x+p\in S\setminus\{x\}=T$. Furthermore, $x=\m(S)=\min_\preceq\big((T\setminus\{0\}
)\cup\{x\}\big)\prec\min_\preceq(T\setminus\{0\}
)=\m(T)$. Conversely, let $x \in \SG(T) \setminus \{f\}$ such that $x + p \in T$ for some $p \in P$. So, $S=T \cup \{x\}\in \CaS(P,f)$. If $x \prec \m(T)$, then $S\setminus\{x\}=T$  with $\m(S) = x$, and thus $S$ is a child of $T$.
\end{proof}

We present an algorithm to compute the set $\CaS(P,f)$ using the tree structure described in Theorem~\ref{thr:treeA(P,f)}.

\begin{algorithm}[H]
\caption{Computing $P$-semigroups with a given Frobenius element.}\label{alg:A(P,f)}
\KwIn{Let $\CaC\subseteq\N^d$ be a non-negative integer cone, $f\in\CaC\setminus\{0\}$, $P\subset \N^d$, and $\preceq$ a monomial order.}
\KwOut{The set  $\CaS(P,f)$.}

$A \leftarrow \{\Delta(f)\}$\;
$X\leftarrow A$\;
\While {$A\ne\emptyset$}{
    $Y \leftarrow \emptyset$\;
    $Z\leftarrow A$\;
    \While {$Z\ne\emptyset$}{
        $T \leftarrow \text{First}(Z)$\;
        $B \leftarrow \{x\in \SG(T)\setminus\{f\}\mid x\prec \m(T) \text{ such that } x+p\in T \text{ for some } p\in P\}$\;
        $Y \leftarrow Y\cup\{T\cup\{x\}\mid x\in B\}$\label{lineaAlg2}\;
        $Z \leftarrow Z\setminus \{T\}$\;
    }
    $A \leftarrow Y$\;
    $X \leftarrow X\cup Y$\;
    }
    \Return{$X$}
\end{algorithm}

To compute the set of special gaps for each $\CaC$-semigroup $S$ arising in Algorithm \ref{alg:A(P,f)}, we consider the approach detailed in \cite[Proposition 5 and 6]{R-T-V_Asemigroup}, using Apéry set $\Ap(S, b) = \{a \in S \mid a - b \in \CaH(S)\}$ with respect to $b \in S \setminus \{0\}$. This finite set satisfies
\[
\Ap(S \cup \{x\}, b) \subseteq \{x\} \sqcup \left( \Ap(S, b) \setminus \{x + b\} \right),
\]
for any $x \in \SG(S)$, which permits the computation of $\SG(S \cup \{x\})$ from $\SG(S)$.

To illustrate the application of the algorithm, we present the following example.

\begin{example}\label{example:Frobenius}
Consider the cone $\CaC$, the set $P$, and the monomial order given in Example \ref{example:treeGenus}, and $f=(9,2)$ as the inputs of Algorithm \ref{alg:A(P,f)}. The elements belonging to $\CaS(P,f)$ are detailed in Figure \ref{fig:treeFrobenius}. In this situation, the label of each non-root node is the special gap added in line \ref{lineaAlg2} of Algorithm \ref{alg:A(P,f)}.
The root $\Delta(f)$ is the $P$-semigroup minimally generated by
\begin{multline*}
\{(9,3),(10,2),(10,3),(11,3),(12,3),(12,4),(13,3),(13,4),(14,3),(14,4),\\
(15,3),(15,4),(15,5),(16,4),(16,5),(17,4),(17,5),(18,4),(18,5),(19,4)\}.
\end{multline*}
\begin{figure}[ht]
\centering
\begin{tikzpicture}[scale=0.8]
\tikzset{level distance=3em}
\Tree
[.$\Delta(f)$
        [.(6,2)
            [.(3,1) ]
        ]
        [.(7,2)
            [.(6,2)
                [.(3,1) ]
            ]
        ]
        [.(8,2)
            [.(5,1) ]
            [.(6,2)
                [.(3,1) ]
                [.(5,1)
                    [.(3,1) ]
                ]
            ]
            [.(7,2)
                [.(4,1) ]
                [.(5,1) ]
                [.(6,2)
                    [.(3,1) ]
                    [.(4,1)
                        [.(3,1) ]
                    ]
                    [.(5,1)
                        [.(3,1) ]
                    ]
                ]
            ]
        ]
]
\end{tikzpicture}

\caption{The tree of $G\left(\CaS(P,f)\right)$}
\label{fig:treeFrobenius}
\end{figure}
\end{example}

\section{The elements of $\CaS(P)$ with a given multiplicity}\label{Sec5}

Let $m\in \CaC\setminus\{0\}$. Let $\CaS(P)_m$ be the set formed by all the $P$-semigroups with multiplicity $m$. The main objective in this section is to compute the set $\CaS(P)_m$. First, we study the finiteness of the set $\CaS(P)_m$, distinguishing several cases.

Set $p\in P \cap \CaC$, and consider the $P$-semigroup $\widetilde{S}_a=\{0\}\cup \langle m, p \rangle \cup \{x\in \CaC \mid x\succ a\}$ for some $a\in \CaC$ with $a\succ m$. Note that if $p\prec m$, there exists $k\in \N$ such that $kp\succ m$, and we define $\widetilde{S}_a'=\widetilde{S}_a\setminus\{p, 2p, \ldots ,(k-1)p\}$, which is an element of $\CaS(P)_m$.
Otherwise, that is $p\succeq m$, then $\widetilde{S}_a'= \widetilde{S}_a\in \CaS(P)_m$.
For both cases, we have that
\[
\bigcup_{\substack{ a \in \mathcal{C}\\ a\succ m}} \widetilde{S}'_a \subseteq \CaS(P)_m.
\]
Observe that, when $\langle m, p \rangle$ is not a $\CaC$-semigroup, then  $\CaS(P)_m$ contains an infinite union of different $P$-semigroups with multiplicity $m$. So, analyzing the finiteness of $\CaS(P)_m$ is equivalent to determine when  $\langle m, p \rangle $ is a $\CaC$-semigroup.
Therefore, for those cones $\CaC\subset\N^d$ with at least $d\geq 3$ extreme rays, it is known that  $\langle m, p \rangle $ is not a $\CaC$-semigroup (see \cite{R-T-V_Asemigroup}). For $d=2$, $\langle m, p \rangle $ is a $\CaC$-semigroup if and only if $\langle m, p \rangle=\CaC$ (see \cite{diaz2022characterizing}). For $d=1$, that is, for numerical semigroups, we present the following result.

\begin{proposition}
The set $\CaS(P)_m$ is finite if and only if $\langle m, p \rangle$ is a numerical semigroup for every $p\in P$.
\end{proposition}
\begin{proof}
Suppose that there exists a natural number $p\in P$ such that $\langle m, p \rangle$ is not a numerical semigroup, and let us see that $\CaS(P)_m$ is infinite. By \cite[Lemma 2.1]{libroRosales}, we have that $\gcd(m,p)=t\ne 1$. Observe that $\{0\}\cup \{m+kt\mid k\in\N\} \cup \{x\in \N \mid x\geq am\}$ is an element of  $\CaS(P)_m$ for every non-zero natural number $a$.  Hence, $\CaS(P)_m$ is infinite.

Conversely, assume that $S\in \CaS(P)_m$. Then, $m$ and $m+p$ belong to $S$ for some $p\in P$, and $\langle m, m+p \rangle\subseteq S$. Consequently,
\[
\CaS(P)_m \subseteq \bigcup_{p\in P}\big\{S \text{ numerical semigroup} \mid \langle m, m+p \rangle\subseteq S\big\}.
\]
Each $\langle m, m+p \rangle$ is a numerical semigroup, since  $\langle m, p \rangle$ is a numerical semigroup. Therefore, $\CaS(P)_m$ is contained in a finite union of finite sets, and thus $\CaS(P)_m$ is finite.
\end{proof}

Following analogous arguments to that developed in the previous sections, we define the graph $G\big(\CaS(P)_m\big)$ as follows: the vertex set is $\CaS(P)_m$, and a pair $(S,T)\in \CaS^2(P)_m$ is an edge if and only if $T=S\cup\{\Fb(S)\}$. From this definition, it is immediate that $\Fb(S)\succ m$.

The following result describes the structure of the graph $G\big(\CaS(P)_m\big)$.

\begin{theorem}\label{thr:treemultiplicity}
The graph $G\big(\CaS(P)_m\big)$ is a tree with root $S_m=\{x\in \CaC \mid x\succeq m\}\cup\{0\}$.    Moreover,  the set of children of any $T\in \CaS(P)_m$ is given by
\[
\{ T \setminus \{a\}\in \CaS(P)_m \mid a \in \msg(T)\setminus\{m\},\ a \succ \Fb(T)\}.
\]
\end{theorem}

\begin{proof}
The proof is analogous to that of Theorem~\ref{thr:tree}. The sequence obtained by adjoining the Frobenius element in each term, restricted by the condition $\Fb(S)\succ m$.
\end{proof}

When $\CaS(P)_m$ is finite, the recursive application of Theorem~\ref{thr:treemultiplicity} provides an algorithmic method to compute such a set.

\begin{example}
Let $P=\{1,2\}\subset \N$. Figure \ref{fig:finito} shows the elements of the set $\CaS(P)_3$.
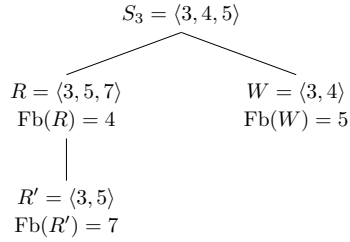
\begin{figure}[ht]
\centering
\begin{tikzpicture}[scale=0.8]
\tikzset{level distance=5em, sibling distance=5em}
\Tree
[.$S_3=\langle 3,4,5\rangle$
        [.\shortstack{$R=\langle 3,5,7\rangle$ \\ $\Fb(R)=4$}
            [.\shortstack{$R'=\langle 3,5\rangle$ \\ $\Fb(R')=7$} ]
        ]
        [.\shortstack{$W=\langle 3,4\rangle$ \\ $\Fb(W)=5$} ]
]
\end{tikzpicture}
\caption{The set $\CaS(P)_3$}
\label{fig:finito}
\end{figure}
\end{example}

On the contrary, if $\CaS(P)_m$ is not finite, then it is necessary to restrict the computation by fixing the genus. Based on Proposition~\ref{prop:caracREMOVEelement} and Theorem~\ref{thr:treemultiplicity}, Algorithm~\ref{ComputeCaA(P)fixgenus} can be adapted to compute $\CaS(P)_m$ up to genus $g$ with the following modifications:
\begin{itemize}
    \item Line \ref{lineROOT} replace the root $\CaC$ by $S_m$.
    \item Line \ref{lineMSG} replace by $A \leftarrow \{x\in \msg(T)\setminus\{m\}\mid x\succ \Fb(T)\}$.
\end{itemize}

\begin{example}\label{example:treeGenus}
Let $\CaC$ be the non-negative integer cone given in Example \ref{example:treeGenus}. Consider $P=\{(1,4),(6,2),(7,2)\}$ and the graded lexicographical order. The tree given in Figure \ref{fig:treeMultiplicity} shows the elements in the set $\CaS(P)_{(4,1)}$ up to genus 3. In this example, the label of a non-root node corresponds to the minimal generator that has been removed.

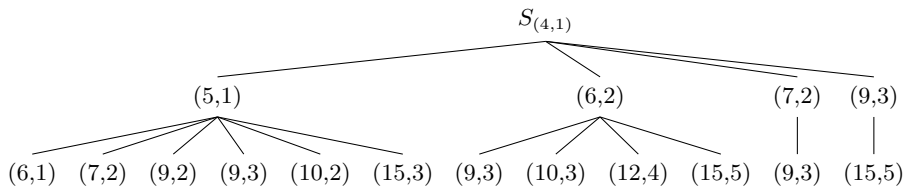
\begin{figure}[ht]
\resizebox{\textwidth}{!}{%
\begin{tikzpicture}
\tikzset{level distance=3em}
\small
\Tree
[.$S_{(4,1)}$
        [.(5,1)
            [.(6,1) ]
            [.(7,2) ]
            [.(9,2) ]
            [.(9,3) ]
            [.(10,2) ]
            [.(15,3) ]
        ]
        [.(6,2)
            [.(9,3) ]
            [.(10,3) ]
            [.(12,4) ]
            [.(15,5) ]
        ]
        [.(7,2)
            [.(9,3) ]
        ]
        [.(9,3)
            [.(15,5) ]
        ]
]

\end{tikzpicture}
}

\caption{The tree of the elements in $\CaS(P)_m$ up to genus 3}
\label{fig:treeMultiplicity}
\end{figure}
\end{example}

\subsection*{Funding}

The first and last two authors are partially supported by grant PID2022-138906NB-C21, funded by MICIU/AEI/10.13039/501100011033 and by ERDF/EU.

Consejería de Universidad, Investigación e Innovación de la Junta de Andalucía project ProyExcel\_00868 and research group FQM343 also partially supported the last three authors.

This publication and research have been partially funded by INDESS (Research University Institute for Sustainable Social Development), Universidad de Cádiz, Spain.

\subsection*{Acknowledgments}
We are grateful to Daniel Escánez-Expósito for his comments and notes.

\subsection*{Author information}

\noindent
J. I. Farr\'{a}n. Departamento de Matemática Aplicada, Universidad de Valladolid, Campus de Segovia.
Member of the IMUVa (Instituto de Matemáticas de la Universidad de Valladolid).
E-40005 Segovia (Segovia, Spain). E-mail: jifarran@uva.es.

\medskip

\noindent
J. C. Rosales. Departamento de \'{A}lgebra, Universidad de Granada, E-18071 Granada, (Granada, Spain).
E-mail: jrosales@ugr.es.

\medskip

\noindent
R. Tapia-Ramos. Departamento de Matemáticas, Universidad de Cádiz, E-11406 Jerez de la Frontera (Cádiz, Spain).
E-mail: raquel.tapia@uca.es.

\medskip

\noindent
A. Vigneron-Tenorio. Departamento de Matemáticas/INDESS (Instituto Universitario para el Desarrollo Social Sostenible), Universidad de C\'adiz, E-11406 Jerez de la Frontera (C\'{a}diz, Spain).
E-mail: alberto.vigneron@uca.es.

\subsection*{Data Availability}
The authors confirm that the data supporting some findings of this study are available within it.

\subsection*{Conflict of Interest}
The authors declare that they have no conflict of interest.

\end{document}